\documentclass[12pt]{jloganal} 

\newtheorem{thm}{Theorem}[section]
\newtheorem{lemma}[thm]{Lemma}
\newtheorem{prop}[thm]{Proposition}
\newtheorem{cor}[thm]{Corollary}

\theoremstyle{definition}

\newtheorem{df}[thm]{Definition}

\newtheorem{nrmk}[thm]{Remark}

\theoremstyle{remark}

\newtheorem*{rmks}{Remarks}

\theoremstyle{prcl}
\newtheorem*{prclaim}{Proclaim}

\renewcommand{\labelenumi}{\rm (\theenumi)}

\newenvironment{renumerate}
        {\renewcommand{\labelenumi}{\rm (\roman{enumi})}
         \begin{enumerate}}
        {\end{enumerate}}

\newenvironment{lenumerate}[2]
        {\renewcommand{\labelenumi}{\rm (#1\theenumi)}
         \begin{enumerate}{\setcounter{enumi}{#2}}}
        {\end{enumerate}}

\newcounter{flexnummark}

\renewcommand{\cl}{\mathrm{cl}}
\DeclareMathOperator{\fr}{fr}
\DeclareMathOperator{\bd}{bd}

\newcommand{\rest}[1]{\!\!\upharpoonright_{#1}}

\newcommand{\into}{\longrightarrow}

\renewcommand{\bar}{\overline}

\newcommand{\set}[1]{\left\{#1\right\}}

\newcommand{\NN}{\mathbb{N}}

\newcommand{\RR}{\mathbb{R}}

\newcommand{\curly}[1]{\mathcal{#1}}

\newcommand{\B}{\curly{B}}
\renewcommand{\C}{\curly{C}}
\renewcommand{\D}{\curly{D}}

\newcommand{\K}{\curly{K}}

\newcommand{\N}{\curly{N}}
\renewcommand{\P}{\curly{P}}

\renewcommand{\R}{\curly{R}}
\renewcommand{\S}{\curly{S}}
\newcommand{\T}{\curly{T}}

\newcommand{\Ran}{\RR_{\text{an}}}

\numberwithin{equation}{section}

\allowdisplaybreaks[2]

\title {Hausdorff limits of Rolle leaves}

\author {Jean-Marie Lion and Patrick Speissegger}

\address {IRMAR, Universit\'e de Rennes I, Campus de
  Beaulieu, 35042 Rennes cedex, France}

\email{jean-marie.lion@univ-rennes1.fr}

\address {Department of Mathematics \& Statistics, McMaster
  University, 1280 Main Street West, Hamilton, Ontario L8S 4K1,
  Canada}

\email {speisseg@math.mcmaster.ca}

\date{\today}



\thanks {Supported by the CNRS of France and NSERC of Canada grant
  RGPIN 261961.}

\begin{document}

\begin{abstract}
  Let $\R$ be an o-minimal expansion of the real field.  We introduce
  a class of Hausdorff limits, the $T^\infty$-limits over $\R$, that
  do not in general fall under the scope of Marker and Steinhorn's
  definability-of-types theorem.  We prove that if $\R$ admits
  analytic cell decomposition, then every $T^\infty$-limit over $\R$
  is definable in the pfaffian closure of $\,\R$.
\end{abstract}

\maketitle

\section*{Introduction}

We fix an o-minimal expansion $\R$ of the real field.  In this paper,
we study $T^\infty$-limits over $\R$ as defined in Section \ref{def}
below; they generalize the pfaffian limits over $\R$ introduced in
\cite[Section 4]{Lion:2009cf}.  Pfaffian limits over $\R$ are
definable in the pfaffian closure $\P(\R)$ of $\R$
\cite{Speissegger:1999nt}, by the variant of Marker and Steinhorn's
definability-of-types theorem \cite{Marker:1994jw} found in van den
Dries \cite[Theorem 3.1]{Dries:2005ls} and \cite[Theorem
1]{Lion:2004jv}.  The $T^\infty$-limits over $\R$ considered here do
not seem to fall under the scope of these theorems, as explained in
Section \ref{def} below.  Nevertheless, $T^\infty$-limits were used by
Lion and Rolin \cite{Lion:1998ay} to establish the o-minimality of the
expansion of $\Ran$ by all Rolle leaves over $\Ran$ of codimension
one.

To state our results, we work in the setting of
\cite[Introduction]{Lion:2009cf}; in particular, recall that a set $W
\subseteq \RR^n$ is a \textbf{Rolle leaf over $\R$} if there exists a
nested Rolle leaf $(W_0, \dots, W_k)$ over $\R$ such that $W = W_k$.

First, we obtain the following generalization of \cite[Th\'eor\`eme
1]{Lion:1998ay}.

\begin{prclaim}[Theorem A] 
  Let $\N(\R)$ be the expansion of $\R$ by all Rolle leaves over $\R$.
  \begin{enumerate}
  \item There is an o-minimal expansion $T^\infty(\R)$ of $\N(\R)$ in
    which every $T^\infty$-limit over $\R$ is definable.
  \item Let $M \subseteq \RR^n$ be a bounded, definable $C^2$-manifold
    and $d$ be a definable and integrable nested distribution on $M$.
    Let $K \subseteq \RR^n$ be a $T^\infty$-limit obtained from $d$.
    Then $\dim K \le \dim d$.
  \end{enumerate}
\end{prclaim}

The question then arises how $T^\infty(\R)$ relates to the pfaffian
closure $\P(\R)$ of $\R$.  Indeed, we do not know in general if
$T^\infty(\R)$ is interdefinable with $\N(\R)$ or $\P(\R)$, or if
$T^\infty(T^\infty(\R))$ is interdefinable with $T^\infty(\R)$.  Based
on \cite{Lion:2009cf}, we can answer such questions under an
additional hypothesis:

\begin{prclaim}[Theorem B]
  Assume that $\R$ admits analytic cell decomposition.  
  \begin{enumerate}
  \item Every $T^\infty$-limit over $\P(\R)$ is definable in $\P(\R)$.
  \item The structures $T^\infty(\R)$ and $\P(\R)$ are interdefinable;
    in particular, $T^\infty(\R)$ and $T^\infty(T^\infty(\R))$ are
    interdefinable.
  \end{enumerate}
\end{prclaim}

We view the combination of Theorems A(2) and B(1) as a non-first order
extension of \cite[Theorem 3.1]{Dries:2005ls} and \cite[Theorem
1]{Lion:2004jv}.

Our proofs of these theorems rely heavily on terminology and notation
introduced in \cite[Introduction and Section 2]{Lion:2009cf}; we do
not repeat the respective definitions here.  We prove Theorem A in
Section \ref{omin} below using the approach of
\cite{Speissegger:1999nt}, but based on a straightforward adaptation
of some results of \cite[Section 4]{Lion:2009cf} to $T^\infty$-limits
carried out in Section \ref{limits} below.  Theorem B then follows by
adapting \cite[Proposition 7.1]{Lion:2009cf} to $T^\infty$-limits and
using \cite[Proposition 10.4]{Lion:2009cf}; the details are given in
Section \ref{theorem_B}.

\section{The definitions}  \label{def}

Let $M \subseteq \RR^n$ be a bounded, definable $C^2$-manifold of
dimension $m$.  We adopt the terminology and results found in
\cite[Introduction and Section 2]{Lion:2009cf}, and we let $d = (d_0,
\dots, d_k)$ be a definable and integrable nested distribution on $M$.

A sequence $(V_\iota)_{\iota \in \NN}$ of integral manifolds of $d_k$
is a \textbf{$T^\infty$-sequence of integral manifolds of $d$} if
there are a core distribution $e = (e_0, \dots, e_l)$ of $d$, a
sequence $(W_\iota)$ of Rolle leaves of $e$ and a definable family
$\B$ of closed integral manifolds of $d_{k-l}$ such that each
$V_\iota$ is an admissible integral manifold of $d$ with core
$W_\iota$ corresponding to $e$ and definable part in $\B$
corresponding to $W_\iota$, as defined in \cite[Definition
4.1]{Lion:2009cf}.  

In this situation, we call $(W_\iota)$ the \textbf{core sequence} of
the sequence $(V_\iota)$ \textbf{corresponding to $e$} and $\B$ a
\textbf{definable part} of the sequence $(V_\iota)$
\textbf{corresponding to $(W_\iota)$}.

\begin{rmks}
  \begin{enumerate}
  \item We think of the core sequence of $(V_\iota)$ as representing
    the ``non-definable part'' of $(V_\iota)$.  If $W_\iota = W_1$ for
    all $\iota$, then $(V_\iota)$ is an admissible sequence of
    integral manifolds of $d$ as defined in \cite[Definition
    4.3]{Lion:2009cf}.
  \item Let $(V_\iota)$ be a $T^\infty$-sequence of integral manifolds
    of $d$.  Then there is a $T^\infty$-sequence $(U_\iota)$ of
    integral manifolds of $(d_0, \dots, d_{k-1})$ such that $V_\iota
    \subseteq U_\iota$ for $\iota \in \NN$.
  \end{enumerate}
\end{rmks}

Let $(V_\iota)$ be a $T^\infty$-sequence of integral manifolds of $d$.
If $(V_\iota)$ converges to $K$ in $\K_n$ (the space of all compact
subsets of $\RR^n$ equipped with the Hausdorff metric), we call $K$ a
\textbf{$T^\infty$-limit over $\R$}.  In this situation, we say that
$K$ \textbf{is obtained from} $d$, and we put
\begin{equation*}
  \deg K := \min\set{\deg f:\ K \text{ is obtained from } f}.
\end{equation*}

\begin{rmks}
  \begin{lenumerate}{}{2}
  \item It is unknown whether the family of all Rolle leaves of $e$ is
    definable in $\P(\R)$ \footnote{For instance, a positive answer to
      this question for all $e$ definable in $\P(\R)$ would imply the
      second part of Hilbert's 16th problem.}.  As a consequence,
    contrary to the situation described by \cite[Lemma
    4.5]{Lion:2009cf} for pfaffian limits over $\R$, the variant of
    Marker and Steinhorn's definability-of-types theorem
    \cite{Marker:1994jw} found in \cite[Theorem 3.1]{Dries:2005ls} and
    \cite[Theorem 1]{Lion:2004jv} does not apply; in particular, we do
    not know in general whether a $T^\infty$-limit over $\R$ is
    definable in $\P(\R)$.
  \item If $W_\iota = W_1$ for all $\iota$, then $K$ is a pfaffian
    limit over $\R$ as introduced in \cite[Definition
    4.4]{Lion:2009cf}.
  \end{lenumerate} 
\end{rmks}

\section{Towards the proof of Theorem A}
\label{limits}

Let $M \subseteq \RR^n$ be a definable $C^2$-manifold of dimension
$m$.  

\subsection*{Pfaffian fiber cutting}
We fix a finite family $\Delta = \{d^1, \dots, d^q\}$ of definable
nested distributions on $M$; we write $d^p = (d^p_0, \dots,
d^p_{k(p)})$ for $p = 1, \dots, q$.  As in \cite[Section
3]{Lion:2009cf}, we associate to $\Delta$ the following set of
distributions on $M$:
\begin{equation*}
  \D_\Delta := \set{d^0_0 \cap d^1_{k(1)} \cap \cdots \cap
    d^{p-1}_{k(p-1)} \cap d^p_j:\ p = 1, \dots, q \text{ and } j = 0,
    \dots, k(p)}, 
\end{equation*}
where we put $d^0_0 := g_M$.  If $N$ is a $C^2$-submanifold of $M$
compatible with $\D_\Delta$, we let $d^{\Delta,N} =
\left(d^{\Delta,N}_0, \dots, d^{\Delta,N}_{k(\Delta,N)}\right)$ be the
nested distribution on $N$ obtained by listing the set $\set{g^N:\ g
  \in \D_\Delta}$ in order of decreasing dimension.  In this
situation, if $V_p$ is an integral manifold of $d^p_{k(p)}$, for $p=1,
\dots, q$, then the set $N \cap V_1 \cap \cdots \cap V_q$ is an
integral manifold of $d^{\Delta,N}_{k(\Delta,N)}$.

Let $A \subseteq M$ be definable.  For $I \subseteq \{1, \dots, q\}$
we put $\Delta(I) := \{d^p:\ p \in I\}$.

\begin{lemma}
  \label{cell_cutting}
  Let $I \subseteq \{1, \dots, q\}$.  Then there is a finite
  partition $\P$ of definable $C^2$-cells contained in $A$ such that
  $\P$ is compatible with $\D_{\Delta(J)}$ for every $J \subseteq \{1,
  \dots, q\}$ and
  \begin{renumerate}
  \item $\dim d^{\Delta(I),N}_{k(\Delta(I),N)} = 0$ for every $N
    \in \P$;
  \item whenever $V_p$ is a Rolle leaf of $d^p$ for $p \in I$, every
    component of $A \cap \bigcap_{p \in I} V_p$ intersects some
    cell in $\P$.
  \end{renumerate}
\end{lemma}

\begin{proof}
  By induction on $\dim A$; if $\dim A = 0$, there is nothing to do,
  so we assume $\dim A > 0$ and the corollary is true for lower values
  of $\dim A$.  By \cite[Proposition 2.2]{Lion:2009cf} and the
  inductive hypothesis, we may assume that $A$ is a $C^2$-cell
  compatible with $\D_{\Delta(J)}$ for $J \subseteq \{1, \dots, q\}$.
  Thus, if $\dim d^{\Delta(I),A}_{k(\Delta(I),A)} = 0$, we are done;
  otherwise, we let $\phi$ and $B$ be as in \cite[Lemma
  3.1]{Lion:2009cf} with $\Delta(I)$ in place of $\Delta$.

  Let $V_p$ be a Rolle leaf of $d^p$ for each $p$; it suffices to
  show that every component of $X:= A \cap \bigcap_{p \in I} V_p$
  intersects $B$.  However, since $d^{\Delta(I),A}_{k(\Delta(I),A)}$
  has dimension, $X$ is a closed, embedded submanifold of $A$.  Thus,
  $\phi$ attains a maximum on every component of $X$, and any point in
  $X$ where $\phi$ attains a local maximum belongs to $B$.
\end{proof}

\begin{cor}  
  \label{cell_fiber_cutting}
  Let $d$ be a definable nested distribution on $M$ and $m \le n$.
  Then there is a finite partition $\P$ of $C^2$-cells contained in
  $A$ such that for every Rolle leaf $\,V$ of $d$, we have
  \begin{equation*}
    \Pi_m(A \cap V) = \bigcup_{N \in \P} \Pi_m(N \cap V)
  \end{equation*}
  and for every $N \in \P$, the set $N \cap V$ is a
  submanifold of $\,U$, $\Pi_m\rest{(N \cap V)}$ is an immersion and
  for every $n' \le n$ and every strictly increasing $\lambda:\{1,
  \dots, n'\} \into \{1, \dots, n\}$, the projection
  $\Pi_{\lambda}\rest{(N \cap V)}$ has constant rank.
\end{cor}

\begin{proof}
  Apply Lemma \ref{cell_cutting} with $q:= n+1$, $d^p:= \ker
  dx_p$ for $p=1, \dots, n$, $d^q:= d$ and $I:= \{1, \dots, m,n+1\}$.
\end{proof}

\subsection*{$T^\infty$-limits}
We assume that $M$ has a definable $C^2$-carpeting function $\phi$,
and we let $d = (d_0, \dots, d_k)$ be a definable distribution on $M$
with core distribution $e = (e_0, \dots, e_l)$. 

First, we reformulate \cite[Proposition 4.7]{Lion:2009cf}.  We adopt
the notation introduced before \cite[Proposition 4.6]{Lion:2009cf} and
note that the $q$ in \cite[Remark 4.2]{Lion:2009cf} can be chosen
independent of the particular $W$.

\begin{prop}
  \label{frontier_limit}
  Let $(V_\iota)$ be a $T^\infty$-sequence of integral manifolds of
  $d$ with core sequence $(W_\iota)$, and assume that $K':= \lim_\iota
  \fr V_\iota$ exists.  Then $K'$ is a finite union of
  $T^\infty$-limits obtained from $d^{M'}$ with core sequences among
  $\big((W_\iota)^{M'}_1\big)_\iota, \dots,
  \big((W_\iota)^{M'}_q\big)_\iota$.
\end{prop}

\begin{proof}
  Exactly as for \cite[Proposition 4.7]{Lion:2009cf}, except for
  replacing ``core $W$'' by ``core sequence $(W_\iota)$'' and ``core
  $W^{M'}_p$'' by ``core sequence $\big((W_\iota)^{M'}_p\big)$''.
\end{proof}

Second, as we do not know yet whether $T^\infty$-limits are definable
in an \hbox{o-minimal} structure, we work with the following notion of
dimension (see also van den Dries and Speissegger \cite[Section
8.2]{Dries:1998xr}): we call $N \subseteq \RR^n$ a
\textbf{$C^0$-manifold of dimension $p$} if $N \neq \emptyset$ and
each point of $N$ has an open neighbourhood in $N$ homeomorphic to
$\RR^p$; in this case $p$ is uniquely determined (by a theorem of
Brouwer), and we write $p = \dim(N)$.  Correspondingly, a set $S
\subseteq \RR^n$ \textbf{has dimension} if $S$ is a countable union of
$C^0$-manifolds, and in this case put
\begin{equation*}
  \dim(S) := \begin{cases}
    \max\{\dim(N): N \subseteq S \text{ is a $C^0$-manifold} \}
    &\text{if } S \neq \emptyset \\ -\infty &\text{otherwise}.
  \end{cases}
\end{equation*}
It follows (by a Baire category argument) that, if $S = \bigcup_{i \in
  \NN} S_i$ and each $S_i$ has dimension, then $S$ has dimension and
$\dim(S) = \max\{\dim(S_i): i \in \NN\}$.  Thus, if $N$ is a
$C^1$-manifold of dimension $p$, then $N$ has dimension in the sense
of this definition and the two dimensions of $N$ agree.

\begin{cor}
  \label{limit_dimension}
  In the situation of \cite[Lemma 1.5]{Lion:2009cf}, the set
  $\lim_\iota V_\iota \setminus \lim_\iota \fr V_\iota$ is either
  empty or has dimension $p$. \qed
\end{cor}

Therefore, we replace \cite[Lemma 4.5]{Lion:2009cf} by

\begin{prop}
  \label{limit_dim}
  Let $K$ be a $T^\infty$-limit obtained from $d$.  Then $K$ has
  dimension and satisfies $\dim K \le \dim d$.
\end{prop}

\begin{proof}
  Let $(V_\iota)$ be a $T^\infty$-sequence of integral manifolds of
  $d$ such that $K = \lim_\iota V_\iota$.  We proceed by induction on
  $\dim d$.  If $\dim d = 0$, then \cite[Corollary
  3.3(2)]{Lion:2009cf} gives a uniform bound on the cardinality of
  $V_\iota$, so $K$ is finite.  So assume $\dim d > 0$ and the
  corollary holds for lower values of $\dim d$.

  By \cite[Proposition 2.2 and Remark 4.2]{Lion:2009cf}, we may assume
  that $M$ is a definable $C^2$-cell; in particular, there is a
  definable $C^2$-carpeting function $\phi$ on $M$.  For each $\sigma
  \in \Sigma_n$, let $M_{\sigma,2n}$ be as before \cite[Lemma
  1.3]{Lion:2009cf} with $d_k$ in place of $d$.  Then by that lemma,
  $M = \bigcup_{\sigma \in \Sigma} M_{\sigma,2n}$ and each
  $M_{\sigma,2n}$ is an open subset of $M$.  Hence $d$ is compatible
  with each $M_{\sigma,2n}$, and after passing to a subsequence if
  necessary, we may assume that $K_\sigma = \lim_\iota (V_\iota \cap
  M_{\sigma,2n})$ exists for each $\sigma$.  It follows that $K =
  \bigcup_{\sigma \in \Sigma_n} K_\sigma$, so by \cite[Lemma
  1.3(2)]{Lion:2009cf}, after replacing $M$ with each
  $\sigma^{-1}(M_{\sigma,2n})$, we may assume that $d_k$ is
  $2n$-bounded.  Passing to a subsequence again, we may assume that
  $K' := \lim_\iota \fr V_\iota$ exists as well.  Then by Corollary
  \ref{limit_dimension}, the set $K \setminus K'$ is either empty or
  has dimension $\dim d$.  By Proposition \ref{frontier_limit} and the
  discussion before \cite[Proposition 4.6]{Lion:2009cf}, the set $K'$
  is a finite union of $T^\infty$-limits obtained from a definable
  nested distribution $d'$ on a definable manifold $M'$ that satisfies
  $\deg d' \le \deg d$ and $\dim d' < \dim d$.  So $K'$ has dimension
  with $\dim K' < \dim d$ by the inductive hypothesis, and the
  proposition is proved.
\end{proof}

\begin{df} 
  \label{proper_NPL} 
  A $T^\infty$-limit $K \subseteq \RR^n$ obtained from $d$ is
  \textbf{proper} if $\dim K = \dim d$.
\end{df}

\begin{cor}
  \label{proper_limit}
  Let $K \subseteq \RR^n$ be a $T^\infty$-limit obtained from $d$.
  Then $K$ is a finite union of proper $T^\infty$-limits over $\R$ of
  degree at most $\deg d$. 
\end{cor}

\begin{proof}
  We proceed by induction on $\dim d$; as in the previous proof, we
  assume $\dim d > 0$ and the corollary holds for lower values of
  $\dim d$.  If $\dim K = \dim d$, we are done, so assume that $\dim K
  < \dim d$.  Also as in the previous proof, we now reduce to the case
  where $d_k$ is $2n$-bounded and $K' := \lim_\iota \fr V_\iota$
  exists.  Then Corollary \ref{limit_dimension} implies that $K = K'$,
  so the corollary follows from Proposition \ref{frontier_limit} and
  the inductive hypothesis.
\end{proof}

Finally, $T^\infty$-limits over $\R$ are well behaved with respect to
intersecting with closed definable sets.  To see this, define
$\mathbf{M}:= M \times (0,1)$ and write $(x,\epsilon)$ for the typical
element of $\mathbf{M}$ with $x \in M$ and $\epsilon \in (0,1)$.  We
consider the components of $d$ as distributions on $\mathbf{M}$ in the
obvious way, and we set $\mathbf{d}_0 := g_{\mathbf{M}}$,
$\mathbf{d}_1:= d\epsilon \rest{\mathbf{M}}$ and $\mathbf{d}_{1+i} :=
d_i \cap \mathbf{d}_1$ for $i= 1, \dots, k$ and put $\mathbf{d}:=
(\mathbf{d}_0, \dots, \mathbf{d}_{1+k})$.  Moreover, whenever $e$ is a
core distribution of $d$, we similarly define a corresponding core
distribution $\mathbf{e} = (\mathbf{e}_0, \dots, \mathbf{e}_{1+l})$ of
$\mathbf{d}$.  In this situation, for every Rolle leaf $W$ of $e$ and
every $\epsilon \in (0,1)$, the set $\mathbf{W} := W \times
\{\epsilon\}$ is a Rolle leaf of $\mathbf{e}$.

\begin{prop}  
  \label{intersection_of_NPL}
  Let $K$ be a $T^\infty$-limit obtained from $d$, and let $C
  \subseteq \RR^n$ be a definable closed set.  Then there is a
  definable open subset $\mathbf{N}$ of $\,\mathbf{M}$ and there are
  $q \in \NN$ and $T^\infty$-limits $K_1, \dots, K_q \subseteq
  \RR^{n+1}$ obtained from $\mathbf{d}^{\mathbf{N}}$ such that $K \cap
  C = \Pi_n(K_1) \cup \cdots \cup \Pi_n(K_q)$.
\end{prop}

\begin{proof}[Sketch of proof]
  For $\epsilon > 0$ put $T(C,\epsilon):= \set{x \in \RR^n:\ d(x,C) <
    \epsilon}$.  Note first that $K \cap C = \bigcap_{\epsilon > 0}
  \big(K \cap T(C,\epsilon)\big)$, and the latter is equal to
  $\lim_{\epsilon \to 0} \big(K \cap T(C,\epsilon)\big)$ in the sense
  of \cite[Definition 1.7]{Lion:2009cf}.  Next, let $(V_\iota)$ be a
  $T^\infty$-sequence of integral manifolds of $d$ such that $K =
  \lim_\iota V_\iota$.  Then for every $\epsilon > 0$, there is a
  subsequence $(\iota(\kappa))$ of $(\iota)$ such that the sequence
  $(V_{\iota(\kappa)} \cap T(C,\epsilon))$ converges to some compact
  set $K_\epsilon$.  Note that $K_\epsilon \cap T(C,\epsilon) = K \cap
  T(C,\epsilon)$, since $T(C,\epsilon)$ is an open set.

  Fix a sequence $(\epsilon_\kappa)$ of positive real numbers
  approaching $0$, and for each $\kappa$, choose $\iota(\kappa)$ such
  that $d(V_{\iota(\kappa)} \cap T(C,\epsilon_\kappa),
  K_{\epsilon_\kappa}) < \epsilon_\kappa$. Passing to a subsequence if
  necessary, we may assume that $\lim_\kappa K_{\epsilon_\kappa}$
  and $\lim_\kappa \big(V_{\iota(\kappa)} \cap
  T(C,\epsilon_\kappa)\big)$ exist; note that these limits are then
  equal.  Hence by the above, 
  \begin{align*}
    K \cap C &= \lim_\kappa \big(K \cap T(C,\epsilon_\kappa)\big) =
    \lim_\kappa \big(K_{\epsilon_\kappa} \cap
    T(C,\epsilon_\kappa)\big) \\ &\subseteq \lim_\kappa
    K_{\epsilon_\kappa} = \lim_\kappa \big(V_{\iota(\kappa)} \cap
    T(C,\epsilon_\kappa)\big).
  \end{align*}
  The reverse inclusion is obvious, so $K \cap C = \lim_\kappa
  \big(V_{\iota(\kappa)} \cap T(C,\epsilon_\kappa)\big)$.  Therefore,
  put $\mathbf{N}:= \set{(x,\epsilon) \in \mathbf{M}:\ d(x,C) <
    \epsilon}$; then $\mathbf{N}$ is an open, definable subset of
  $\mathbf{M}$ and by the above $K \cap C = \lim_\kappa
  (V_{\iota(\kappa)} \cap \mathbf{N}^{\epsilon_\kappa})$, where
  $\mathbf{N}^{\epsilon}:= \set{x \in M:\ (x,\epsilon) \in
    \mathbf{N}}$.  Hence $K \cap C = \lim_\kappa
  \Pi_n\big((V_{\iota(\kappa)} \times \{\epsilon_\kappa\}) \cap
  \mathbf{N}\big)$. Since $\lim_\kappa\epsilon_\kappa = 0$, it follows
  that $K \cap C = \Pi_n \big( \lim_\kappa \big((V_{\iota(\kappa)}
  \times \{\epsilon_\kappa\}) \cap \mathbf{N}\big)\big)$.  Since the
  sequence $\big(V_{\iota(\kappa)} \times \{\epsilon_\kappa\}\big)$ is
  a $T^\infty$-sequence of integral manifolds of $\mathbf{d}$, the
  proposition now follows from \cite[Remark 4.2]{Lion:2009cf}.
\end{proof}

\begin{nrmk}
  \label{uniform_intersection}
  Let $\B$ and $\C$ be two definable families of closed subsets of
  $\RR^n$.  Then the $T^\infty$-limits in the previous proposition
  depend uniformly on $C \in \C$, for all $T^\infty$-limits obtained
  from $d$ with definable part $\B$.  That is, there are $\mu,q \in
  \NN$, a bounded, definable manifold $\mathbf{M} \subseteq
  \RR^{n+\mu+1}$, a definable nested distribution $\mathbf{d}$ on
  $\mathbf{M}$ and a definable family $\mathbf{B}$ of subsets of
  $\RR^{n+\nu+1}$ such that whenever $K$ is a $T^\infty$-limit
  obtained from $d$ with definable part $\B$ and $C \in \C$, there are
  $T^\infty$-limits $K_1, \dots, K_q \subseteq \RR^{n+\nu+1}$ obtained
  from $\mathbf{d}$ with definable part $\mathbf{B}$ such that $K \cap
  C = \Pi_n(K_1) \cup \cdots \cup \Pi_n(K_q)$.  
\end{nrmk}

\section{O-minimality and proof of Theorem A}
\label{omin}

Similar to \cite{Lion:1998ay, Speissegger:1999nt}, we show that all
sets definable in $T^\infty(\R)$ are of the following form:

\begin{df}
  \label{lambda-set}
  A set $X \subseteq \RR^m$ is a \textbf{basic $T^\infty$-set} if
  there exist $n \ge m$, a definable, bounded $C^2$-manifold $M
  \subseteq \RR^n$, a definable nested distribution $d$ on $M$ with
  core distribution $e$ and, for $\kappa \in \NN$, a
  $T^\infty$-sequence $(V_{\kappa,\iota})_\iota$ of integral manifolds
  of $d$ with core sequence $(W_{\kappa,\iota})_{\iota}$ corresponding
  to $e$ and definable part $\B$ independent of $\kappa$, such that:
  \begin{renumerate}
  \item for each $\kappa$, the limit $K_\kappa:= \lim_\iota
    V_{\kappa,\iota}$ exists in $\K_n$;
  \item the sequence $(\Pi_m(K_\kappa))_\kappa$ is increasing and has
    union $X$.
  \end{renumerate}
  In this situation, we say that $X$ is \textbf{obtained from $d$}
  with \textbf{core distribution} $e$ and \textbf{definable part}
  $\B$.  A \textbf{$T^\infty$-set} is a finite union of basic
  $T^\infty$-sets.  We denote by $T^\infty_m$ the collection of all
  $T^\infty$-sets in $\RR^m$ and put $T^\infty := (T^\infty_m)_{m \in
    \NN}$.
\end{df}

\begin{prop}
  \label{finite_comp}
  In the situation of Definition \ref{lambda-set}, there is an $N \in
  \NN$ such that every basic $T^\infty$-set obtained from $d$ with
  core distribution $e$ and definable part $\B$ has at most $N$
  components.  In particular, if $X \subseteq \RR^m$ is a
  $T^\infty$-set and $l \le m$, there is an $N \in \NN$ such that for
  every $a \in \RR^l$ the fiber $X_a$ has at most $N$ components.
\end{prop}

\begin{proof}
  Let $N$ be a bound on the number of components of the sets $W \cap
  B$ as $W$ ranges over all Rolle leaves of $e$ and $B$ ranges over
  $\B$.  Let $X$ be a basic $T^\infty$-set as in Definition
  \ref{lambda-set}.  Then each $V_{\kappa,\iota}$ has at most $N$
  components, so each $K_\kappa$ has at most $N$ components, and hence
  $X$ has at most $N$ components.  Combining this observation with
  Remark \ref{uniform_intersection} yields, for every $T^\infty$-set
  $X \subseteq \RR^m$, a uniform bound on the number of connected
  components of the fibers of $X$.
\end{proof}

\begin{prop}
  \label{definable_sets}
  \begin{enumerate}
  \item Any coordinate projection of a $T^\infty$-limit over $\R$ is a
    $T^\infty$-set.
  \item Every bounded definable set is a $T^\infty$-set.
  \item Let $d$ be a definable nested distribution on $M:= (-1,1)^n$
    and $L$ be a Rolle leaf of $d$.  Then $L$ is a $T^\infty$-set.
  \end{enumerate}
\end{prop}

\begin{proof}
  (1) is obvious.  For (2), let $C \subseteq \RR^n$ be a bounded,
  definable cell.  By cell decomposition, it suffices to show that $C$
  is a $T^\infty$-set.  Let $\phi$ be a definable carpeting function on
  $C$.  Then $C = \bigcup_{i = 1}^\infty
  \cl\big(\phi^{-1}((1/i,\infty))\big)$, so let $\mathbf{C} :=
  \set{(x,r) \in C \times (0,1):\ \phi(x)>r}$ and put $\mathbf{d}_1:=
  \ker dr \rest{\mathbf{C}}$ and $\mathbf{d}:= (g_{\mathbf{C}},
  \mathbf{d}_1)$.  Then for $r>0$, the set $\mathbf{C}^r =
  \phi^{-1}((r,\infty)) \times \{r\}$ is an admissible integral
  manifold of $\mathbf{d}$ with core $\mathbf{C}$ and definable
  part $\mathbf{C}^r$, so $\cl(\mathbf{C}^r)$ is a $T^\infty$-limit
  obtained from $\mathbf{d}$.

  (3) Let $\phi$ be a carpeting function on $M$.  Then $$L =
  \bigcup_{i=1}^\infty \cl\big(L \cap \phi^{-1}((1/i,\infty))\big),$$
  so we let $\mathbf{M} := \set{(x,r) \in M \times (0,1):\ \phi(x)>r}$
  and put $\mathbf{d}_0:= g_{\mathbf{M}}$, $\mathbf{d}_1 := \ker dr
  \rest{\mathbf{M}}$, $\mathbf{d}_{1+i} := \mathbf{d}_1 \cap d_i$ for
  $i=1, \dots, k$ and $\mathbf{d} := (\mathbf{d}_0,
  \dots,\mathbf{d}_{1+k})$.  Let $L_1, \dots, L_q$ be the components
  of $(L \times (0,1)) \cap \mathbf{M}$; note that each $L_p$ is a
  Rolle leaf of $\mathbf{d}$.  Thus for $r>0$ and each $p$, the set
  $L_p \cap \phi^{-1}((r,\infty))$ is an admissible integral manifold
  of $\mathbf{d}$ with core $L_p$ and definable part $\mathbf{M}^r
  = \phi^{-1}((r,\infty)) \times \{r\}$.
\end{proof}

\begin{prop}
  \label{closure_properties}
  The collection of all $T^\infty$-sets is closed under taking finite
  unions, finite intersections, coordinate projections, cartesian
  products, permutations of coordinates and topological closure.
\end{prop}

\begin{proof}
  Closure under taking finite unions, coordinate projections and
  permutations of coordinates is obvious from the definition and the
  properties of nested pfaffian sets over $\R$.  

  For topological closure, let $X \subseteq \RR^m$ be a basic
  $T^\infty$-set with associated data as in Definition
  \ref{lambda-set}.  Then $$\cl(X) = \lim_\kappa \Pi_m(K_\kappa) =
  \Pi_m(\lim_\kappa \lim_\iota V_{\kappa,\iota}) = \Pi_m(\lim_\kappa
  V_{\kappa,\iota(\kappa)})$$ for some subsequence
  $(\iota(\kappa))_\kappa$, so $\cl(X)$ is a $T^\infty$-set by
  Proposition \ref{definable_sets}(1).

  For cartesian products, let $X_1 \subseteq \RR^{m_1}$ and $X_2
  \subseteq \RR^{m_2}$ be basic $T^\infty$-sets, and let $M^i \subseteq
  \RR^{n_i}$, $d^i = (d^i_0, \dots, d^i_{k^i})$, $e^i = (e^i_0, \dots,
  e^i_{l^i})$ and $\big(V^i_{\iota,\kappa}\big)$ be the data
  associated to $X_i$ as in Definition \ref{lambda-set}, for $i=1,2$.
  We assume that both $M^1$ and $M^2$ are connected; the general case
  is easily reduced to this situation.  Define
  \begin{equation*}
    \mathbf{M} := \set{(x,y,u,v):\ (x,u) \in M^1 \text{ and } (y,v)
      \in M^2},
  \end{equation*}
  where $x$ ranges over $\RR^{m_1}$, $y$ over $\RR^{m_2}$, $u$ over
  $\RR^{n_1-m_1}$ and $v$ over $\RR^{n_2-m_2}$.  We interpret $d^i$
  and $e^i$ as sets of distributions on $\mathbf{M}$ correspondingly,
  for $i=1,2$, and we define $\mathbf{d}:= (d^1_0, \dots, d^1_{k^1},
  d^1_{k^1} \cap d^2_1, \dots, d^1_{k^1} \cap d^2_{k^2})$ and
  $\mathbf{e}:= (e^1_0, \dots, e^1_{l^1}, e^1_{l^1} \cap e^2_1, \dots,
  e^1_{l^1} \cap e^2_{l^2})$.  Since $M^1$ and $M^2$ are connected,
  each set $$V_{\kappa,\iota}:= \set{(x,y,u,v):\ (x,u) \in
    V^1_{\kappa,\iota} \text{ and } (y,v) \in V^2_{\kappa,\iota}}$$ is
  an admissible integral manifold of $\mathbf{d}$ with core
  distribution $\mathbf{e}$.  It is now easy to see that for each
  $\kappa$, the limit $K_\kappa:= \lim_\iota V_{\kappa,\iota}$ exists
  in $\K_{n_1+n_2}$, and that the sequence
  $\big(\Pi_{k_1+k_2}(K_\kappa)\big)$ is increasing and has union $X_1
  \times X_2$.

  For intersections, let $X_1, X_2 \subseteq \RR^m$ be basic
  $T^\infty$-sets.  Then $X_1 \cap X_2 = \Pi_k((X_1 \times X_2) \cap
  \Delta)$, where $\Delta:= \set{(x,y) \in \RR^m \times \RR^m:\ x_i =
    y_i \text{ for } i=1, \dots, m}$.  Therefore, we let $X \subseteq
  \RR^m$ be a basic $T^\infty$-set and $C \subseteq \RR^m$ be closed
  and definable, and we show that $X \cap C$ is a $T^\infty$-set.  Let
  the data associated to $X$ be as in Definition \ref{lambda-set}, and
  let $\mathbf{M}$, $\mathbf{d}$ and $\mathbf{e}$ be associated to
  that data as before Proposition \ref{intersection_of_NPL}.  Let also
  $\mathbf{N}$ be the open subset of $\mathbf{M}$ given by that
  proposition with $C':= C \times \RR^{n-m}$ in place of $C$.  Then by
  that proposition, there is a $q \in \NN$ such that for every
  $\kappa$ the set $K_\kappa \cap C'$ is the union of the projections
  of $T^\infty$-limits $K^1_\kappa, \dots, K^q_\kappa$ obtained from
  $\mathbf{d}^{\mathbf{N}}$.  Note that each $K^j_\kappa$ is the limit
  of a $T^\infty$-sequence of integral manifolds of
  $\mathbf{d}^{\mathbf{N}}$ with core distribution
  $\mathbf{e}^{\mathbf{N}}$.  Replacing each sequence
  $\big(K^j_\kappa\big)$ by a (possibly finite) subsequence if
  necessary, we may assume that each sequence
  $\big(\Pi_m(K^j_\kappa)\big)$ is increasing.  Then each $X_j:=
  \bigcup_\kappa K^j_\kappa$ is a basic $T^\infty$-set and $X \cap C =
  X_1 \cup \cdots \cup X_q$.
\end{proof}

\begin{prop}
  \label{boundary}
  Let $X \subseteq \RR^m$ be a $T^\infty$-set.  Then $\bd(X)$ is
  contained in a closed $T^\infty$-set with empty interior.
\end{prop}

\begin{proof}
  Let the data associated to $X$ be given as in Definition
  \ref{lambda-set} and write $d = (d_0, \dots, d_k)$.  Note
  that $$\bd(X) \subseteq \lim_\kappa \bd(\Pi_m(K_\kappa)).$$ Fix an
  arbitrary $\kappa$; since $\Pi_m(K_\kappa) = \lim_\iota
  \Pi_m(V_{\kappa,\iota})$ we may assume, by Corollary
  \ref{cell_fiber_cutting}, \cite[Remark 4.2]{Lion:2009cf} and after
  replacing $M$ if necessary, that $\Pi_k\rest{d_k}$ is an immersion
  and has constant rank $r \le m$; in particular,
  $\dim(V_{\kappa,\iota}) \le m$.  If $r < m$, then each
  $\Pi_m(K_\kappa)$ has empty interior by Proposition
  \ref{limit_dimension}, so $$\lim_\kappa \bd(\Pi_m(K_\kappa)) =
  \lim_\kappa \Pi_m(K_\kappa) = \Pi_m (\lim_\kappa K_\kappa) =
  \Pi_m(\lim_\kappa V_{\kappa,\iota(\kappa)})$$ for some subsequence
  $(\iota(\kappa))$, and we conclude by Propositions \ref{limit_dim}
  and \ref{definable_sets}(1) in this case.  So assume that $r = m$;
  in particular, $\Pi_m(V_{\kappa,\iota})$ is open for every $\kappa$
  and $\iota$.  In this case, since $M$ is bounded, we have
  $\bd(\Pi_m(K_\kappa)) \subseteq \Pi_m(\lim_\iota \fr
  V_{\kappa,\iota})$ for each $\kappa$.  Hence $$\lim_\kappa
  \bd(\Pi_m(K_\kappa)) \subseteq \Pi_m(\lim_\kappa \lim_\iota \fr
  V_{\kappa,\iota}) = \Pi_m(\lim_\kappa \fr
  V_{\kappa,\iota(\kappa)})$$ for some subsequence $(\iota(\kappa))$.
  Now use Propositions \ref{frontier_limit} and
  \ref{definable_sets}(1).
\end{proof}

Following \cite{Wilkie:1999wi} and \cite{Lion:1998ay}, and proceeding
exactly as in \cite[Corollary 3.11 and Proposition
3.12]{Speissegger:1999nt} using the previous propositions, we
obtain:

\begin{prop}
  \label{complement}
  \begin{enumerate}
  \item Let $X \subseteq \RR^m$ be a $T^\infty$-set, and let $\,1 \leq
    l \leq m$.  Then the set $B:= \set{a \in \RR^l:\ \cl(X_a) \neq
      \cl(X)_a}$ has empty interior.
  \item Let $X \subseteq [-1,1]^m$ be a $T^\infty$-set.  Then $[-1,1]^m
    \setminus X$ is also a $T^\infty$-set.  \qed
  \end{enumerate}
\end{prop}

For $m \in \NN$, let $\T_m$ be the collection of all $T^\infty$-sets $X
\subseteq I^m$.

\begin{cor} 
  \label{tinfomin}
  The collection $\T:= (\T_m)_{m \in \NN}$ forms an o-minimal
  structure on $I$.  \qed
\end{cor}

\begin{proof}[Proof of Theorem A]
  For each $m$, let $\tau_m: \RR^m
  \into (-1,1)^m$ be the (definable) homeomorphism given by
  $$\tau_m(x_1, \dots, x_m):= \left(\frac {x_1}{1+x_1^2}, \dots, \frac
    {x_m}{1+x_m^2}\right),$$ and let $\S_m$ be the collection of sets
  $\tau_m^{-1}(X)$ with $X \in \T_m$.  By Corollary \ref{tinfomin},
  the collection $\S = \S:= (\S_m)_m$ gives rise to an o-minimal
  expansion $T^\infty(\R)$ of $\R$.  By Proposition
  \ref{definable_sets}(2), every definable set is definable in
  $T^\infty(\R)$.  But if $L$ is a Rolle leaf of a definable nested
  distribution $d$ on $\RR^n$, then $\tau_n(L)$ is a Rolle leaf of the
  pullback $(\tau_n^{-1})^*d$.  It follows from Proposition
  \ref{definable_sets}(3) that $\tau_n(L) \in \T_n$, so $L$ is
  definable in $T^\infty(\R)$.  Therefore, $\N(\R)$ is a reduct of
  $T^\infty(\R)$ in the sense of definability.
\end{proof}

\section{Proof  of Theorem B}  \label{theorem_B}

First, we establish \cite[Proposition 7.1]{Lion:2009cf} with
``$T^\infty$-limit'' and ``$T^\infty(\R)$'' in place of ``pfaffian
limit'' and ``$\P(\R)$''.  To do so, we proceed exactly as in
\cite{Lion:2009cf}, making the following additional changes.
\begin{itemize}
\item[(B1)] Replacing ``admissible sequence'' with ``$T^\infty$-sequence'',
  we obtain corresponding versions of Lemma 4.8, Remark 4.9,
  Proposition 4.11, Corollary 4.13 and Proposition 5.3 in
  \cite{Lion:2009cf}.
\item[(B2)] Using (B1), we obtain the corresponding version
  of \cite[Proposition 7.1]{Lion:2009cf}.
\end{itemize}

Second, assuming that $\R$ admits analytic cell decomposition, (B2)
and \cite[Proposition 10.4]{Lion:2009cf} imply that every
$T^\infty$-limit over $\R$ is definable in $\N(\R)$; in particular,
$T^\infty(\R)$ and $\N(\R)$ are interdefinable.  Hence, by
\cite[Corollary 1]{Lion:2009cf}, $T^\infty(\R)$ and
$\P(\R)$ are interdefinable.  Replacing once more $\R$ by $\P(\R)$,
Theorem B is now proved.

\end{document}